\documentclass[reqno,oneside,11pt]{amsart}

\usepackage[T1]{fontenc}
\usepackage{tikz}
\usetikzlibrary{plotmarks, shapes.geometric, backgrounds}
\usepackage{color,soul}
\usepackage{xcolor}
\usepackage[utf8]{inputenc}
\usepackage{lmodern}
\usepackage[english]{babel}
\usepackage{comment}
\usepackage[unicode, pdftex,
            pdfauthor={Wodson Mendson/João Paulo},
            pdftitle={Non-algebraicity of foliations via reduction modulo $2$},
            pdfsubject={foliations, algebraic geometry},
            pdfkeywords={foliations on surfaces, p-divisor},
            pdfproducer={Latex with hyperref},
            pdfcreator={Overleaf}]{hyperref}
\usepackage{amsmath}
\usepackage{amsfonts}
\usepackage{amssymb}
\usepackage{amsthm}
\usepackage{MnSymbol}
\usepackage{graphicx}
\usepackage{tabularx}

\newtheorem{theorem}{Theorem}

\newtheorem{thm}{Theorem}[section]
\newtheorem{lemma}[thm]{Lemma}
\newtheorem{cor}[thm]{Corollary}
\newtheorem{prop}[thm]{Proposition}
\theoremstyle{definition}
\newtheorem{OBS}[thm]{Remark}

\newtheorem{dfn}[thm]{Definition}

\DeclareMathOperator{\field}{k}

\DeclareMathOperator{\spm}{\textbf{Spm}}

\usepackage{tikz-cd}
\usepackage{todonotes}

\title{Non-algebraicity of foliations via reduction modulo $2$}

\author[Jo\~ao Paulo Figueredo]{Jo\~ao Paulo Figueredo}

\author[Wodson Mendson]{Wodson Mendson}

\address{J. P. Figueredo, Universidade Federal Fluminense, Instituto de Matemática e Estatística.
Rua Alexandre Moura 8, São Domingos, 24210-200 Niterói RJ, Brazil.}
\email{joaoplf@id.uff.br}

\address{W. Mendson, Universidade Federal Fluminense, Instituto de Matemática e Estatística.
Rua Alexandre Moura 8, São Domingos, 24210-200 Niterói RJ, Brazil.}
\email{oliveirawodson@gmail.com}

\begin{document}


\begin{abstract}  Motivated by the Jouanolou foliation problem, we investigate the non-algebraicity of foliations by curves on $\mathbb{P}^2_{\mathbb{C}}$. We present a criterion to show that such a foliation has no algebraic invariant curves, using a method of reduction modulo $2$. Finally, using this criterion we give a new proof that the Jouanolou foliation of odd degree has no algebraic invariant curves. We also present other classes of foliations without algebraic invariant curves.
\end{abstract}

\maketitle

\setcounter{tocdepth}{1}
\tableofcontents

\section{Introduction} 

At the end of the XIX century, Darboux, Poincaré and others initiated the study of orbits of vector fields on the complex projective plane. This was the beginning of the theory of holomorphic foliations. Since then, several developments have been made on this subject, the study of the existence of algebraic orbits of such vector fields being one example. We can name the Poincaré problem, which asks for criteria to decide when the orbits of a complex vector field on the complex projective plane are algebraic. This is shown to be equivalent to bounding the degree of a generic orbit (see \cite{poincareproblem}).

In general, given a foliation on the complex projective plane, it is unlikely that it has an algebraic invariant curve. To be more precise, considering the set ${\rm Fol}_d(\mathbb{P}^2_\mathbb{C})$ of all holomorphic foliations with a fixed degree $d \geq 2$ on $\mathbb{P}^2_\mathbb{C}$, one can endow ${\rm Fol}_d(\mathbb{P}^2_\mathbb{C})$ with the structure of a quasi-projective variety. It turns out that the foliations having no algebraic invariant curve form a countable intersection of open subsets of ${\rm Fol}_d(\mathbb{P}^2_\mathbb{C})$. It is also dense, if one shows the existence of a foliation of degree $d$ having no algebraic invariant curve. This is achieved via the Jouanolou foliation, which is given by the following vector field:
\[v_d = (xy^d-1)\partial_x - (x^d - y^{d+1})\partial_y\]
on $\mathbb{A}^2_\mathbb{C}$.

Jouanolou \cite{jouanolou} showed that this foliation does not have algebraic invariant curves using the structure of the group of automorphisms of $\mathbb{P}^2_\mathbb{C}$ acting on the foliation. More recently, the second named author has showed in \cite{mendson2023arithmetic} the same result, when the degree of the foliation is odd and $d \not \equiv 1 \mod 3$, using a reduction modulo $p$ method. This was achieved considering the reduction modulo $2$ of the foliation. 

For a foliation by curves $\mathcal{F}$ in characteristic $p$, if $v$ is a vector field tangent to $\mathcal{F}$, then we may ask if $v^p$, which is also a vector field since the characteristic is $p$, is also tangent to $\mathcal{F}$. If this is the case, we say that $\mathcal{F}$ is $p$-closed. The $p$-closed foliations on $\mathbb{P}^2_{\overline{\mathbb{F}}_p} $ correspond to finite inseparable morphisms $\mathbb{P}^2_{\overline{\mathbb{F}}_p} \rightarrow Y$ of degree $p$, where $Y$ is normal (see \cite{rudakov1976inseparable}). If $\mathcal{F}$ is not $p$-closed, the vanishing locus of $v\wedge v^p$ defines a divisor $\Delta_\mathcal{F}$, called the $p$-divisor of the reduction modulo $p$ of $\mathcal{F}$.

In the case of the Jouanolou foliation reduced modulo $2$, \cite{mendson2023arithmetic} shows that it is not $2$-closed and that its $2$-divisor $\Delta$ is irreducible. Then an argument comparing the possible degree of an invariant curve with the degree of the $2$-divisor shows that it is not possible to have an algebraic invariant curve (see proof of Theorem \ref{main}).

Motivated by the Jouanolou foliation problem, in this paper we treat the problem of determining the non existence of algebraic invariant curves for holomorphic foliations on the complex projective plane. We will pursue this problem via the same approach used in \cite{mendson2023arithmetic}, i.e. using reduction modulo $p$ methods. In this direction, our main result is the following (see Theorem \ref{main}; compare with \cite{bost2001algebraic}):

\begin{theorem} \label{thmA}
     Let $\mathcal{F}$ be a foliation on $\mathbb{P}_{\mathbb{C}}^{2}$ of degree $d > 1$ with the following properties:
\begin{enumerate}
    \item $\mathcal{F}$ is defined over a number field;
    \item $\mathcal{F}$ is not dicritical;
    \item $\mathcal{F}$ has good reduction modulo $2$;
    \item The reduction modulo $2$ of $\mathcal{F}$ has irreducible $2$-divisor.
\end{enumerate}
Then $\mathcal{F}$ has no algebraic solutions.
\end{theorem}

The Jouanolou foliation of odd degree satisfies these hypotheses, and we improve the aforementioned result in \cite{mendson2023arithmetic} (see Corollary \ref{Jouuu}). Finally, by experimental tests using Singular (see Appendix), we can modify the vector field defining the Jouanolou foliation, and we give examples of other classes of foliations on $\mathbb{P}^2_{\mathbb{C}}$ which have no algebraic invariant curve. We also use our methods to give a proof that a foliation given in \cite{alcantara2020foliations} has only one algebraic invariant curve.

\subsection{Organization of the paper} In Section \ref{notacoes} we fix the notations that will be used in the paper. In Section \ref{pdivisoremP2} we recall the definition of foliation on a smooth algebraic surface and the notion of the $p$-divisor of foliations in characteristic $p>0$, and we discuss some topics related to the reduction modulo $p$ of foliations.  In Section \ref{ring}, we investigate the ring of definition of algebraic invariant curves under a foliation. In the last section, Section \ref{algtwo}, we prove Theorem \ref{thmA} and we construct new families, via methods of reduction modulo two, of foliations in the complex projective plane that have no algebraic invariant curves.

\section{Notation} \label{notacoes}

\begin{itemize}

    \item $\Delta_{\mathcal{F}}$ = the $p$-divisor associated to a foliation $\mathcal{F}$;
    
    \item \textbf{$p$-factor of an effective divisor $E$} = an effective divisor of the form $pD$ with $pD\leq E$;

    \item $\spm(R) = $ the collection of maximal ideals of a domain $R$;

    \item $\mathbb{F}_{\mathfrak{q}} = R/\mathfrak{q}$, the residue field of $\mathfrak{q}\in \spm(R)$;
    
    \item $\mathbb{Z}[\mathcal{F}]=$ the $\mathbb{Z}$-algebra obtained by adjunction of all coefficients which appear in $\mathcal{F}$ and the variety on which $\mathcal{F}$ is defined;

    \item $\mathcal{F}_\mathfrak{p} = $ the reduction modulo $\mathfrak{p}$ of $\mathcal{F}$;
    
\end{itemize}

\section{Foliations in positive characteristic and the \texorpdfstring{$p$}{p}-divisor}\label{pdivisoremP2}

We refer to \cite{MR3328860} for the basic theory of holomorphic foliations, and to \cite{mendson2023arithmetic} for the basic theory of foliations in positive characteristic. We review the latter here. Let $\field$ be an algebraically closed field of characteristic $p > 0$. Let $S$ be a regular algebraic surface over $\field$. A foliation $\mathcal{F}$ on $S$ is given by a sub-line-bundle $T_{\mathcal{F}} \subset T_{S}$ such that $T_{S}/T_{\mathcal{F}}$ is torsion-free. The inclusion $T_\mathcal{F} \subset T_{S}$ induces a non-zero global section $v \in {\rm H}^0(S, T_{S} \otimes T_{\mathcal{F}}^*)$. We say that $\mathcal{F}$ is $p$-closed if $v \wedge v^p = 0$.

\begin{OBS}
    By duality, one may define $\mathcal{F}$ by a sub linebundle $N_\mathcal{F}^* \subset \Omega_S^1$, where $N_\mathcal{F} = (T_S/T_\mathcal{F})^{**}$. Thus we get a non-zero global section $\omega \in {\rm H}^0(S, \Omega_S^1 \otimes N_\mathcal{F})$, and $\mathcal{F}$ will be $p$-closed iff $\iota_{v^p}(\omega) = 0$.
\end{OBS}

If $\mathcal{F}$ is not $p$-closed, let $\Delta_{\mathcal{F}}$ be the divisor given by the vanishing of $v \wedge v^p$. We call $\Delta_{\mathcal{F}}$ the $p$-divisor associated to $\mathcal{F}$. We have $\mathcal{O}_S(\Delta_\mathcal{F}) \cong N_\mathcal{F} \otimes T_\mathcal{F}^{\otimes -p}$. In the case $S = \mathbb{P}^2_{\field}$, the $p$-divisor of a foliation of degree $d$ has degree $p(d-1)+d+2$.

\subsection{Reduction modulo \texorpdfstring{$p$}{p} of foliations}

Let $S$ be a complex regular quasi-projective surface, and let $\mathcal{F}$ be a holomorphic foliation on $S$. Let $v_1, \dots, v_n$ be local vector fields defining $\mathcal{F}$ on $S$, and $F_1,\dots,F_m$ be polynomials which define $S$ in some open set of a projective space. We define $\mathbb{Z}[\mathcal{F}]$ as the smallest sub-$\mathbb{Z}$-algebra of $\mathbb{C}$ which the coefficients of the $v_i$ and the $F_j$ are defined. This is independent of the choice of the $v_i$ and the $F_j$. 

Let $\mathfrak{p}\in \spm(\mathbb{Z}[\mathcal{F}])$ be a maximal ideal and let $\mathbb{F}_{\mathfrak{p}} = \mathbb{Z}[\mathcal{F}]/\mathfrak{p}$.  Note that $\mathbb{F}_{\mathfrak{p}}$ is a finite field by \cite[Tag00GB]{stacks-project}, and, in particular, it has characteristic $p>0$.

\begin{dfn} 
    We define the \textbf{reduction modulo $\mathfrak{p}$} of $\mathcal{F}$ as the foliation $\mathcal{F}_{\mathfrak{p}}$ on $S_\mathfrak{p}$ given by the saturation of $T_{\mathcal{F}} \otimes {\rm Spec}(\overline{\mathbb{F}_{\mathfrak{p}}})$ in $T_{S_\mathfrak{p}}$, where $S_\mathfrak{p} = S \times_{{\rm Spec}(\mathbb{Z}[\mathcal{F}])} {\rm Spec}(\overline{\mathbb{F}_{\mathfrak{p}}})$.
\end{dfn}

\section{The ring of definition of invariant curves}\label{ring}

If a foliation $\mathcal{F}$ on $\mathbb{P}^2_{\mathbb{C}}$ is defined over a ring $R \subset \overline{\mathbb{Q}}$, and has an invariant algebraic curve $C$, then it is not necessarily true that $C$ is defined over $R$. In this section, we show, however, that in this case it is possible to find another curve which is invariant by $\mathcal{F}$ and defined over $R$. 

The following proposition is well-known, and we give its proof for the reader's convenience.

\begin{prop}\label{ext} 
    Let $R$ be a domain and $f, g \in R[x_1,\ldots,x_n]$ non-constant polynomials without common factors. Let $\field$ be a field containing $R$. Then $f\otimes \field$ and $g\otimes \field$ do not have common factors in $\field[x_1,\ldots,x_n]$. 
\end{prop}
\begin{proof} 
    Suppose by contradiction that there exists $H$ a common factor of $f\otimes \field$ and $g\otimes \field$. Assume, WLOG, that $\deg_{x_{n}}(h)>0$. In particular, $x_n$ appears in $f\otimes \field$ and $g\otimes \field$. Since $f$ and $g$ do not have a common factor in $(R[x_1,\ldots,x_{n-1}])[x_n]$ Gauss Lemma implies that $f$ and $g$ do not have a common factor in $K[x_n]$, where $K$ is the fraction field of $R[x_1,\ldots,x_{n-1}]$. Since $K[x_n]$ is PID there exists a relation
    $$
        Af+Bg = 1
    $$
    where $A,B \in K$. Clearing denominators, it follows that 
    $$
        af+bg = c
    $$
    for some $a,b,c \in R[x_1,\ldots,x_{n-1}]$. Now, scalar extension to $\field$ implies that $H$ divides $c$ in $\field[x_1,\ldots,x_n]$. This implies
    $$
        0 = \deg_{x_{n}}(c)\geq \deg_{x_n}(H)>0
    $$
    a contradiction.
\end{proof}

\begin{cor} 
    Let $R$ be a UFD of characteristic zero and $f \in R[x_1,\ldots,x_n]$ be non-constant and irreducible. Then, for every field $K$ which contains $R$, the polynomial $f\otimes K$ is reduced in $\field[x_1,\ldots,x_n]$.
\end{cor}

\begin{proof} 
    We know that $f$ is reduced if and only if $f$ is square-free. On the other hand, $f$ is square-free if and only if for some canonical derivation $D_i := \partial_{x_i}$ the polynomials $f$ and $D_i(f)$ do not have common factors and the result follows from Proposition \ref{ext}   
\end{proof}

Now, we apply Proposition \ref{ext} in the context of foliations. See \cite[Theorem 2.2]{coutinho2018bounding} and \cite[Proposition]{man1997rational}.
\begin{prop}\label{coefi} 
    Let $\mathcal{F}$ be a foliation on the complex projective plane defined by a projective $1$-form defined over a domain $R\subset \overline{\mathbb{Q}}$. Let $C = \{F = 0\}$ be an irreducible algebraic curve that is $\mathcal{F}$-invariant. Then, there exists a reduced algebraic curve that is $\mathcal{F}$-invariant defined by an irreducible polynomial over $R$.
\end{prop}
\begin{proof}  
    Fix $x_0, x_1, x_2$ homogeneous coordinates of $\mathbb{P}_{\mathbb{C}}^{2}$ and write $F=F_a+F_l$, where $F_a = \sum_{\alpha \in A} a_{\alpha}\textbf{x}^{\alpha}$ and $F_{l} = \sum_{\beta \in B} b_{\beta}\textbf{x}^{\beta}$ with $a_{\alpha}\in \overline{\mathbb{Q}}$ for every $\alpha \in A$ and $b_{\beta} \in \mathbb{C}-\overline{\mathbb{Q}}$ for every $\beta \in B$. Choose algebraic numbers $\{c_{\beta}\}_{\beta \in B}$ such that $\sum_{\beta \in B} c_{\beta}\textbf{x}^{\beta} \neq -\sum_{\alpha \in A}a_{\alpha}\textbf{x}^{\alpha}$ and consider the specialization homomorphism 
    $$
        ev_{c}\colon \overline{\mathbb{Q}}[\{X_{\beta}\mid \beta \in B\}] \longrightarrow \overline{\mathbb{Q}} \qquad P(\{X_{\beta}\})\mapsto P(\{c_{\beta}\})
    $$
    and define $P= ev_{c}(F)$. Note that $P$ is defined over $\overline{\mathbb{Q}}$ and since $\mathcal{F}$ is defined over $R \subset \overline{\mathbb{Q}}$, then applying $ev_c$ to the algebraic equation giving the $\mathcal{F}$-invariance of $F$, we get an equation giving the $\mathcal{F}$-invariance of $P$; in other words, the curve defined by $P$ is $\mathcal{F}$-invariant. Let $K$ be the fraction field of $R$ and consider $L|K$ the smallest Galois extension where $P$ is defined. Define
    $$
        Q = \prod_{\sigma \in G_{L|K}}P^{\sigma} \in L[x,y,z]
    $$
    where $G_{L|K}$ is the Galois group of $L|K$. The coefficients of $Q$ belong to $K$, and we can assume that $Q$ is defined over $R$, cleaning denominators if necessary. Since $\mathcal{F}$ is defined over $R$, the same argument as in the previous paragraph implies that the curve defined by $Q$ is $\mathcal{F}$-invariant.  Now, let $Q_{l} \in R[x,y,z]$ be an irreducible factor of $Q$. Since $\mathcal{Z}(Q) $ is $\mathcal{F}$-invariant we have $\mathcal{Z}(Q_l)$ $\mathcal{F}$-invariant. Proposition \ref{ext} ensures that $Q_{l}\otimes\mathbb{C}$ is reduced. This finishes the proof.
\end{proof}

\section{Non-algebraicity via characteristic two}\label{algtwo}

In this section we will study the non-algebraicity of foliations on $\mathbb{P}^2_{\mathbb{C}}$ via reduction modulo $2$. More precisely, we will consider non-$2$-closed foliations $\mathcal{F}$ and we will show that if $\Delta_\mathcal{F}$ is irreducible mod $2$ and $\mathcal{F}$ is non-dicritical, then it has no algebraic invariant curve. We use this result to show that the Jouanolou foliation of odd degree $d$ has no algebraic invariant curve. We also use this result to construct classes of foliations on $\mathbb{P}^2_\mathbb{C}$ having at most one algebraic invariant curve.

\begin{dfn}Given a regular projective variety $X \subset \mathbb{P}^n_\mathbb{C}$, defined over a finitely generated $\mathbb{Z}$-algebra $R$, we say that $\mathfrak{p} \in {\rm Spec}(R)$ is a prime of good reduction of $X$ if $X \times_{{\rm Spec}(R)} {\rm Spec}(\overline{R/\mathfrak{p}})$ is regular. Let $\mathcal{F}$ be foliation on $X$ and $\mathfrak{p} \in {\rm Spm}(\mathbb{Z}[\mathcal{F}])$. We say that $\mathfrak{p}$ is a prime of good reduction of $\mathcal{F}$ if $T_{\mathcal{F}_\mathfrak{p}} = T_\mathcal{F} \otimes {\rm Spec}(\overline{\mathbb{F}_\mathfrak{p}})$ and it has rank ${\rm rank}(T_\mathcal{F})$.
\end{dfn}

\begin{dfn} Let $S$ be a complex projective surface, $C$ be an irreducible curve in $S$ and $p \in \mathbb{Z}$ a prime of good reduction of $C$. We say that $C$ is a $p$-factor if $C \times_{{\rm Spec}(R)}{\rm Spec}(\overline{\mathbb{F}_p}) = pD$ for some effective divisor $D$ on $S \times_{{\rm Spec}(R)} {\rm Spec}(\overline{\mathbb{F}_p})$. 
\end{dfn}

\begin{prop} \label{prop:pdivides}
If $p$ does not divide $\deg(C)$ then $C\otimes \mathbb{F}_{p}$ is not a $p$-factor.
\end{prop}

In the following lemma, we show a condition on the curve $C$ to ensure that $C\mod p$ is not a $p$-factor.

\begin{lemma}\label{pfactor} Let $\mathcal{F}$ be a foliation on a complex projective surface $S$ that has good reduction at a prime number $p \in \mathbb{Z}$. If $C$ is an irreducible invariant algebraic curve by $\mathcal{F}$ then $C$ is not a $p$-factor.
\end{lemma}

\begin{proof} Since $p$ is a prime of good reduction of $\mathcal{F}$ we know that $T_\mathcal{F} \otimes {\rm Spec}(\overline{\mathbb{F}_p})$ is saturated in $T_{S_\mathfrak{p}}$. Suppose, by contradiction, that $C \mod p$ is a $p$-factor. Locally around a general $q \in C$, we have that $C$ is the zero locus of $f$ and the vector space $T_q\mathcal{F}$ is generated by $f_y(q) \partial_x - f_x(q) \partial_y$, where $x$ and $y$ are local coordinates of $S$ around $q$. Since $C \mod p$ is a $p$ factor, there exist functions $g$ and $h$ in a  neighborhood of $q$ such that
\[f - g^p = ph.\]
Thus $f_x(q) = p(g^{p-1}(q)g_x(q)) + ph_x(q)$ and $f_y(q) = p(g^{p-1}(q)g_y(q)) + ph_y(q)$. Thus $T_{\mathcal{F}}$ is divisible by $p$ over the general point of $C$.
This implies that $T_\mathcal{F} \otimes {\rm Spec}(\overline{\mathbb{F}_p}) = 0$ generically over $C \times_{{\rm Spec} \mathbb{Z}[\mathcal{F}]} {\rm Spec}(\overline{\mathbb{F}_p})$ and, in particular, $T_\mathcal{F} \otimes {\rm Spec}(\overline{\mathbb{F}_p})$ is not saturated in $T_{S_\mathfrak{p}}$, a contradiction.

\end{proof}

The class of foliations which we will be concerned in our main result is the one of non dicritical foliations, whose definition we recall now.

\begin{dfn}\label{dfn:dicritical}
    Let $\mathcal{F}$ be a foliation on $\mathbb{P}^2_\mathbb{C}$ and let $p \in {\rm sing}(\mathcal{F})$. We say that $p$ is a \textbf{dicritical} singularity if for a local $1$-form $\omega = \sum_{i \geq m} a_i(x,y) dx + b_i(x,y)dy$ defining $\mathcal{F}$ around $p$, where $m$ is the multiplicity of $\mathcal{F}$ at $p$, we have
    \[ a_m(x,y)dx+b_m(x,y)dy = \phi(x,y)(ydx-xdy)\]
    where $\phi(x,y) \in \mathbb{C}[x,y]$.

    We say that $\mathcal{F}$ is \textbf{not dicritical} if it has no dicritical singularities.
\end{dfn}

For the proof of our main theorem, we will use the Carnicer bound:

\begin{thm}[\cite{MR1298714}]\label{carnicer}
    Let $\mathcal{F}$ be a foliation of degree $d$ on $\mathbb{P}^2_\mathbb{C}$ and let $C$ be an irreducible algebraic invariant curve of $\mathcal{F}$. Suppose that there are no dicritical singularities of $\mathcal{F}$ on $C$. Then
    \[ \deg(C) \leq d + 2.\]
\end{thm}

Finally, we are able to prove Theorem \ref{thmA}:

\begin{thm}\label{main} Let $\mathcal{F}$ be a foliation on $\mathbb{P}_{\mathbb{C}}^{2}$ of degree $d > 1$ with the following properties:
\begin{enumerate}
    \item $\mathcal{F}$ is defined over a number field;
    \item $\mathcal{F}$ is not dicritical;
    \item $\mathcal{F}$ has good reduction modulo $2$. 
    \item $\mathcal{F}\otimes \overline{\mathbb{F}_2}$ has irreducible $2$-divisor.
\end{enumerate}
Then $\mathcal{F}$ has no algebraic solutions.
\end{thm}
\begin{proof} Suppose that there is an $\mathcal{F}$-invariant curve $C$ which is defined by an irreducible polynomial $F \in \mathbb{C}[x,y,z]$. By Proposition \ref{coefi} we can assume that $F \in \mathbb{Z}[\mathcal{F}][x,y,z]$ and that it is reduced  and irreducible, and by the Carnicer bound (Theorem \ref{carnicer}) we get $\deg(C)\leq d+2$. Now, since $C$ is $\mathcal{F}$-invariant, we can use Lemma \ref{pfactor}, to conclude that $C\otimes \mathbb{F}_2$ is not a $2$-factor. So, there is a polynomial $Q$, an irreducible factor of $F\otimes \mathbb{F}_2$, such that $\{Q=0\}$ is invariant by the foliation $\mathcal{F} \otimes \mathbb{F}_2$. Since the $2$-divisor $\Delta$ is irreducible, we get: $\Delta = \{Q = 0\}$. In particular,
$$
    d+2\geq \deg(C)\geq \deg(Q) = 3d
$$
and we get $2\geq 2d$ so that $d = 1$, a contraction.    
\end{proof}

Recall that the Jouanolou foliation of degree $d$ on $\mathbb{P}_{\mathbb{C}}^{2}$ is defined by the following vector field
$$
    v = (xy^{d}-1)\partial_{x}-(x^{d}-y^{d+1})\partial_{y}
$$
on $\mathbb{A}^2_\mathbb{C}$.

\begin{cor}\label{Jouuu} Let $d$ be an odd integer. Then the Jouanolou foliation, $\mathcal{J}_d$, of degree $d$ has no algebraic invariant curve.
\end{cor}

\begin{proof} By Definition \ref{dfn:dicritical} it follows that $\mathcal{J}_d$ is nondicritical. By \cite[Proposition 6.3]{mendson2023arithmetic} it follows that the $2$-divisor is irreducible. One checks that the vanishing set of $v$, the vector field defining $\mathcal{J}_d$, is finite in any characteristic, and in particular $\mathcal{J}_d$ has good reduction modulo $2$. Using Theorem \ref{main} we conclude the proof.
\end{proof}

\begin{OBS} 
     A weak version of Theorem \ref{main} appeared in the paper \cite{mendson2022}, where the additional hypothesis $d \not\equiv 1 \mod 3$ is required. 
\end{OBS}

We will construct classes of foliations having irreducible $2$-divisor on $\mathbb{A}^2_\mathbb{C}$. To show irreducibility, we need to understand irreducibility of two variable polynomials. To do so, we will investigate the Newton polytope of these polynomials.

\begin{dfn}
    Let $f \in \field[x_0,\dots,x_n]$, and write $f = \sum_k c_k \mathbf{x}^{a_k}$, where $c_k \in \field$, $\mathbf{x}^{a_k} = x_0^{a_k^0}\dots x_n^{a_k^n}$ and $a_k = (a_k^0,\dots,a_k^n) \in \field^{n+1}$. We define the Newton polytope ${\rm Newt}(f)$ of $f$ as the convex hull of the $a_k$ in $\field^{n+1}$, i.e.
    \[{\rm Newt}(f) = \left\{ \sum_k \alpha_k a_k \in \field^{n+1} \mid \sum_k \alpha_k = 1 \text{ and } \alpha_k \geq 0 \text{ for all } k\right\}.\]
\end{dfn}

The criterion for irreducibility, using the Newton polytope of the polynomial, is presented in the following Lemma. For it, we remind the reader that if $\mathbf{A}$ and $\mathbf{B}$ are two subsets of $\field^{n+1}$, their Minkowski sum is defined as 
\[\mathbf{A} + \mathbf{B} = \{\mathbf{a} + \mathbf{b} \mid \mathbf{a} \in \mathbf{A}, \,\, \mathbf{b} \in \mathbf{B}\}.\]

\begin{lemma} \label{lemma:irreduciblepolynomial}
    Let $f \in \field[x_0,\dots,x_n]$ be a polynomial, and let ${\rm Newt}(f) \subset \mathbb{R}^{n+1}$ be its Newton polytope. If ${\rm Newt}(f)$ cannot be written as the Minkowski sum of two smaller polytopes, then $f$ is irreducible.
\end{lemma}
\begin{proof} See \cite[Lemma 2.1]{MR1816701}.
\end{proof}

The following remark follows directly from the previous Lemma and the definitions. This will be explicitly the tool we will use in our proofs. 

\begin{OBS}[{{\cite{14080}}}] \label{irreduciblitycriterion}
    Let $P$ be a lattice polytope in $\mathbb{R}^2$. Let $v(P)$ be the sequence of vectors $\overrightarrow{p_1 p_2}$, $\overrightarrow{p_2p_3}, \dots,$ $\overrightarrow{p_np_1}$, where we write $\partial P = \{p_1,\dots,p_n\}$. If $v(P)$ cannot be partitioned into two disjoint subsequences, each of which summing to zero, then $P$ cannot be written as the Minkowski of two smaller polytopes.
\end{OBS}

We also recall the following fact about foliations defined by extending forms on the projective plane, which will be needed to decide when a foliation defined on $\mathbb{A}^2_{\mathbb{C}}$, and the extended to $\mathbb{P}^2_{\mathbb{C}}$, has the line at infinity as an invariant curve.

\begin{lemma}\label{lineinf} Let $\omega = \omega_{l}+\omega_{l+1}+\cdots+\omega_{e}$ be a $1$-form on $D_{+}(z)\subset \mathbb{P}_{\field}^{2}$ and consider $\mathcal{F}$ the foliation defined on $\mathbb{P}_{\field}^{2}$ obtained by extension of $\omega$ to $\mathbb{P}_{\field}^{2}$. Let $R = x\partial_{x}+y\partial_{y}$ be the radial vector field.

        \begin{itemize}
            \item  If $i_{R}\omega_e =  0$ then $\mathcal{F}$ has degree $e-1$ and $l_{\infty} = \{z=0\}$ is not invariant;
            \item  If $i_{R}\omega_e \neq  0$ then $\mathcal{F}$ has degree $e$ and $l_{\infty} = \{z=0\}$ is invariant.
        \end{itemize}
        
\end{lemma}

\begin{proof} It follows from a direct computation, or see \cite[Lemma 2]{MR980960}.
\end{proof}

\begin{OBS}
    If $v = a(x,y)\partial_x+b(x,y)\partial_{y} = v_{l}+\cdots+v_{e}$ is a polynomial vector field and the $2$-divisor has degree $3(e-1)$, then $\{z=0\}$ is not invariant. By Lemma \ref{lineinf} the algebraic condition is the following: $v$ has $\{z=0\}$ as an invariant curve if $R\wedge v_{e} \neq 0$.
\end{OBS}

Our goal from now on is to present certain families of foliations on $\mathbb{P}^2_\mathbb{C}$ which will turn out to be non-algebraically integrable, having at most one algebraic invariant curve. We will use the methods developed above, and in particular Theorem \ref{main}. In the first case, presented below, we have a class of foliations with only one algebraic invariant curve.

\begin{thm}\label{thm:claudia2} Let $\field$ be a field of characteristic two, let $d>1$ be an odd integer and $\mathcal{F}$ be the foliation on $\mathbb{P}_{\field}^{2}$ of degree $d^2+d$ defined by the following vector field on $D_{+}(z)$:  
    $$
        v = y^{f(d)}\partial_x+(x+ay^{g(d)}+by^{h(d)}+cy^{s(d)})\partial_y
    $$      
where $abc \not\equiv 0 \mod 2$ and
$$f(d) = d^2+d+1 \qquad g(d) = (d+1)/2 \qquad h(d) = (d^2+d+2)/2 \qquad s(d) = d^2+\frac{d+3}{2}$$
Then, $\mathcal{F}$ is not $2$-closed and its $2$-divisor has the following form:
$$
    \Delta_{\mathcal{F}} = d\{z=0\}+(f(d)-1)\{y=0\}+C
$$
where $C$ is an \textbf{irreducible curve} of degree $2d^2 + d + 3$ given by $g(x,y)$ in the following relation:
$$
       f(x, y) =  y^{f(d)-1}g(x,y)
$$
where 
$$
    g(x,y) = y^{f(d)+1}+B(x,y)(ag(d)y^{g(d)}+bh(d)y^{h(d)}+cs(d)y^{s(d)})+B(x,y)^{2}
$$
and $B(x,y) = x+ay^{g(d)}+by^{h(d)}+cy^{s(d)}$.
\end{thm}
\begin{proof}
    A direct computation shows that $v^2 = v^2(x)\partial_x + v^2(y)\partial_y$, with
    \begin{align*}
        v^2(x) = &\,\, f(d)(xy^{f(d)-1}+ay^{f(d)+g(d)-1}+by^{f(d)+h(d)-1}+cy^{f(d)+s(d)-1})\\
        = & \,\, xy^{f(d)-1}+ay^{f(d)+g(d)-1}+by^{f(d)+h(d)-1}+cy^{f(d)+s(d)-1} \\
        v^2(y) = & \,\, y^{f(d)} + ag(d)xy^{g(d)-1} + a^2g(d)y^{2g(d)-1}+abg(d)y^{g(d)+h(d)-1}\\ & + acg(d)y^{g(d)+s(d)-1}
        + bh(d)xy^{h(d)-1} + abh(d)y^{g(d)+h(d)-1} \\ & +b^2h(d)y^{2h(d)-1}+bch(d)y^{h(d)+s(d)-1}
         + cs(d)xy^{s(d)-1}\\ &+ acs(d)y^{g(d)+s(d)-1}+bcs(d)y^{h(d)+s(d)-1}+c^2s(d)y^{2s(d)-1}
    \end{align*}
using the fact that $d \equiv 1 \mod 2$.
    Thus 
    \begin{align*}
    \frac{v \wedge v^2}{\partial_x \wedge \partial_y} = &\,\,  v(x)v^2(y) - v(y)v^2(x) \\
                = & \,\, y^{f(d) - 1}[y^{f(d)+1}+ ag(d)xy^{g(d)} + bh(d)xy^{h(d)}+ cs(d)xy^{s(d)} \\
                & + ab(g(d)+h(d))y^{g(d)+h(d)} + ac(g(d)+s(d))y^{g(d)+s(d)} \\
                & + bc(h(d)+s(d))y^{h(d)+s(d)} + x^2 + a^2y^{2g(d)} + b^2y^{2h(d)}+c^2y^{2s(d)}]
    \end{align*}
    Taking $g(x,y)$ as in the statement, we see that the above polynomial is just $y^{f(d)-1}g(x,y)$.

    We show now the irreducibility of the curve $C$. For this, we will use Lemma \ref{lemma:irreduciblepolynomial}. Assuming $g(d),h(d),s(d) \equiv 1 \mod 2$, we have 
    \begin{align*}
        {\rm Newt}(g)  =& \,\, \langle (0,f(d)+1), (1,g(d)), (1,h(d)), (1,s(d)), \\
        & (2,0), (0,2g(d)), (0,2h(d)), (0,2s(d)) \rangle .
    \end{align*}
    This is represented as follows:

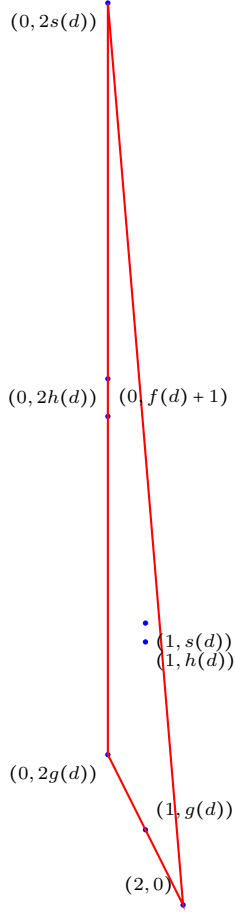
\begin{tikzpicture}[scale=0.5]

\def\d{3} 

\coordinate (A) at (0, \d*\d+\d+1);
\coordinate (B) at (1, \d*0.5 + 0.5);
\coordinate (C) at (1, \d*\d*0.5+\d*0.5+1);
\coordinate (D) at (1, \d*\d*0.5 +\d*0.5+1.5);
\coordinate (H) at (2, 0);
\coordinate (I) at (0, \d+1);
\coordinate (J) at (0, \d*\d+\d+2);
\coordinate (K) at (0, 2*\d*\d + \d + 3);

\foreach \p in {A, B, C, D, H, I, J, K} {
    \fill[blue] (\p) circle (2pt);
}

\draw[red, thick] (I) -- (A) -- (J) -- (K) -- (H) -- (B) -- cycle;

\node[above right] at (A) {{\tiny $(0, f(d)+1)$}};
\node[above right] at (B) {{\tiny $(1, g(d))$}};
\node[below right] at (C) {{\tiny $(1, h(d))$}};
\node[below right] at (D) {{\tiny $(1, s(d))$}};
\node[above left] at (H) {{\tiny $(2, 0)$}};
\node[below left] at (I) {{\tiny $(0, 2g(d))$}};
\node[below left] at (J) {{\tiny $(0, 2h(d))$}};
\node[below left] at (K) {{\tiny $(0, 2s(d))$}};
\end{tikzpicture}

Thus 

\begin{align*}v({\rm Newt}(g)) =& \,\, \{ (0, 2(h(d) - g(d)), (0, f(d) + 1 - 2h(d)), (0,2s(d)-f(d)-1) \\
& (2,-2s(d)), (-1,g(d)), (-1,g(d))\}
\end{align*}

Suppose there exist two subsets $S_1$ and $S_2$ of $v({\rm Newt}(g))$, each one with elements summing to $(0,0)$, such that $v({\rm Newt}(g)) = S_1 \sqcup S_2$. Looking at the first coordinate, $(2,-2s(d))$, $(-1, g(d))$ and $(-1,g(d))$ must be together in one of these sets, say $S_1$. Since their sum is $(0,2g(d)-2s(d))$, now looking at the second coordinate, it must be that $(0,2(h(d) - g(d)))$ also belong to $S_1$. Now, the sum of these is $(0,2h(d) - 2s(d))$, and to cancel out this second coordinate, all both of the other remaining elements must belong to $S_1$, a contradiction. Thus, Lemma \ref{lemma:irreduciblepolynomial} implies that $g$ is irreducible. 

In the case one of $g(d),h(d)$ and $s(d)$ is even, we obtain the same Newton polytope, except possibly by removing $(1,g(d)), (1,h(d))$ and $(1,s(d))$, and adding $(0,g(d)+h(d))$, $(0, g(d) + s(d))$ and $(0,h(d) + s(d))$. Of the possible removable points, the only point in the boundary is $(1,g(d))$. Of the possible additional points, all will be in the boundary.
The following diagram shows ${\rm Newt}(g)$ when all these points are present:

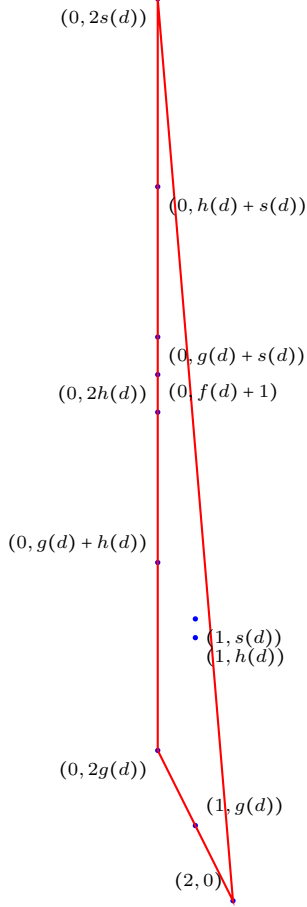
\begin{tikzpicture}[scale=0.5]

\def\d{3} 

\coordinate (A) at (0, \d*\d+\d+1);
\coordinate (B) at (1, \d*0.5 + 0.5);
\coordinate (C) at (1, \d*\d*0.5+\d*0.5+1);
\coordinate (D) at (1, \d*\d*0.5 +\d*0.5+1.5);
\coordinate (E) at (0, \d*0.5+0.5 +\d*\d*0.5+\d*0.5+1);
\coordinate (F) at (0, \d*0.5+1.5 + \d*\d +\d*0.5+1.5);
\coordinate (G) at (0, \d*\d*0.5+\d*0.5+1 + \d*\d +\d*0.5+1.5);
\coordinate (H) at (2, 0);
\coordinate (I) at (0, \d+1);
\coordinate (J) at (0, \d*\d+\d+2);
\coordinate (K) at (0, 2*\d*\d + \d + 3);

\foreach \p in {A, B, C, D, E, F, G, H, I, J, K} {
    \fill[blue] (\p) circle (2pt);
}

\draw[red, thick] (I) -- (E) -- (A) -- (J) -- (F) -- (G) -- (K) -- (H) -- (B) -- cycle;

\node[above right] at (A) {{\tiny $(0, f(d)+1)$}};
\node[above right] at (B) {{\tiny $(1, g(d))$}};
\node[below right] at (C) {{\tiny $(1, h(d))$}};
\node[below right] at (D) {{\tiny $(1, s(d))$}};
\node[above left] at (E) {{\tiny $(0, g(d)+h(d))$}};
\node[below right] at (F) {{\tiny $(0, g(d)+s(d))$}};
\node[below right] at (G) {{\tiny $(0, h(d)+s(d))$}};
\node[above left] at (H) {{\tiny $(2, 0)$}};
\node[below left] at (I) {{\tiny $(0, 2g(d))$}};
\node[below left] at (J) {{\tiny $(0, 2h(d))$}};
\node[below left] at (K) {{\tiny $(0, 2s(d))$}};
\end{tikzpicture}

Considering each one of the possible cases, the same argument as in the previous case shows that $v({\rm Newt}(g))$ cannot be decomposed as the disjoint union of two subsets summing to zero. Again Lemma \ref{lemma:irreduciblepolynomial} implies that $g$ is irreducible.
\end{proof}

\begin{OBS}\label{claudia2} The foliation given by the vector field in Theorem \ref{claudia} appears in the work \cite{alcantara2020foliations} in the proof of \cite[Theorem 2]{alcantara2020foliations} where they
show that it is not algebraically integrable. Note that the foliation has a unique singular point, and this singularity is not dicritical.
\end{OBS}

\begin{thm}\label{claudia} Let $d>1$ be an odd integer and $\mathcal{F}$ be the foliation on $\mathbb{P}_{\mathbb{C}}^{2}$ of degree $d^2+d$ defined by the following vector field on $D_{+}(z)$:  
    $$
        v = y^{f(d)}\partial_x+(x+ay^{g(d)}+by^{h(d)}+cy^{s(d)})\partial_y
    $$      
where $a,b,c \in \mathbb{Z}$, $abc \not\equiv 0 \mod 2$ and
$$f(d) = d^2+d+1 \qquad g(d) = (d+1)/2 \qquad h(d) = (d^2+d+2)/2 \qquad s(d) = d^2+\frac{d+3}{2}$$
Then, $\mathcal{F}$ has $l_{\infty} = \{z = 0\}$ as the unique algebraic invariant curve.
\end{thm}

\begin{proof} Suppose by contradiction that the foliation has an algebraic invariant curve $C$ of degree $e$ in $D_{+}(z)$. By Proposition \ref{coefi}, we can assume that $C$ is defined over $\mathbb{Z}$. Since $\mathcal{F}$ is noncritical (see Remark \ref{claudia2}) Carnicer bound implies $e\leq d+2$. Denote by $C_2$ and $\mathcal{F}_2$ the reduction modulo $2$ of the curve and the foliation, respectively. Note that $\mathcal{F}$ has good reduction at $ p =2 $, since the singular locus of $v$ is finite in any characteristic. 

Since $C\otimes \mathbb{F}_2$ is invariant by $\mathcal{F}$, there is an irreducible factor $Q$ that is invariant. The $2$-divisor determined in Theorem \ref{thm:claudia2} says that 
$$
    \Delta_{\mathcal{F}_2} = (f(d)-1)\{y = 0\}+D
$$
where $D$ is irreducible of degree $2d^2+d+3$. Since $\{y = 0\}$ is not $\mathcal{F}_2$-invariant, we get $Q = D$. So, $$2d^2+d+3=\deg(D) = \deg(Q) \leq \deg(C) \leq d+2.$$ 
In particular, $2d^2 \leq -1$, a contradiction.
\end{proof}

Finally, we present two classes of foliations on $\mathbb{P}^2_\mathbb{C}$, one of which having only one invariant algebraic curve, and the other having none. They were both obtained experimentally, aided by the computer software Singular (see the Appendix), by modifications of the vector field defining the Jouanolou foliation. 

\begin{prop}\label{invII} Let $\field$ be a field of characteristic two and let $a,b,c$ and $u\neq 0$ be elements of $\field$ and for every $d \in \mathbb{Z}$ consider the vector field in $\mathbb{A}_{\field}^{2}$:
$$
    v = (u+xy^{d})\partial_x+(a+bx+cx^{d-1}+y^{d+1})\partial_{y}
$$
If $d \equiv 1 \mod 2$ and $d \geq 5$, then $v$ defines a foliation of degree $d$ on $\mathbb{P}_{\field}^{2}$ with a prime $2$-divisor on $\mathbb{A}^2_{\field}$. 
\end{prop}
\begin{proof} Induction shows that the $2$-divisor is given by $f_{d}(x,y) = f_{d}^{1}(x,y)+f_{d}^{2}(x,y)$ where 
$$
    f_{d}^{1}(x,y) = c^2x^{2d-1}y^{d-1}+cx^{d}y^{2d}+ucx^{d-1}y^{d}+b^{2}x^{3}y^{d-1}+axy^{2d-1}
$$
$$
    f_{d}^{2}(x,y) = ubxy^{d}+a^{2}xy^{d-1}+uy^{2d+1}+uay^{d}+u^2b
$$
We need to show that $f_{d}(x,y) \in \field[x,y]$ is irreducible. First, we determine the Newton polytope of $f_d$:
\begin{align*}{\rm Newt}(f_d)   = & \langle (0,0), (0,d), (0,2d+1), (1,d-1), (1,d), (1,2d-1), (3,d-1), \\ & (d-1,d), (d,2d), (2d-1,d-1)\rangle\end{align*}
The following image shows ${\rm Newt}(f_d)$ for sufficiently big $d$, and its boundary is indicated in red:

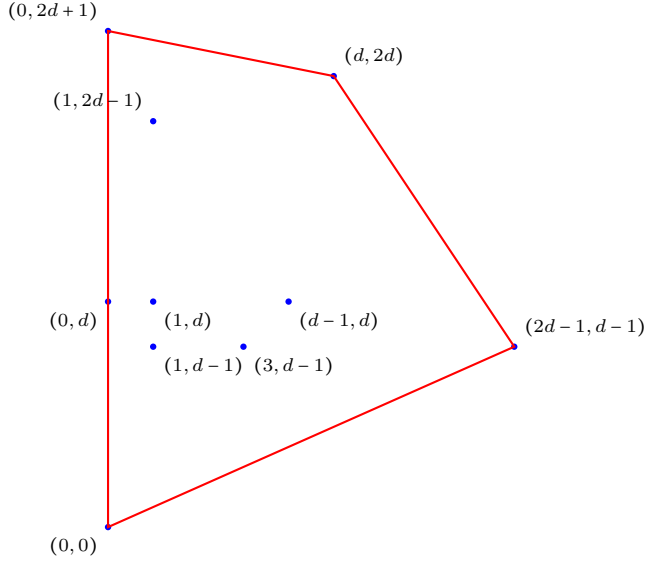
\begin{tikzpicture}[scale=0.6]

\def\d{5}

\coordinate (A) at (2*\d-1, \d-1);
\coordinate (B) at (\d, 2*\d);
\coordinate (C) at (\d-1, \d);
\coordinate (D) at (3, \d-1);
\coordinate (E) at (1, 2*\d-1);
\coordinate (F) at (1, \d);
\coordinate (G) at (1, \d-1);
\coordinate (H) at (0, 2*\d+1);
\coordinate (I) at (0, \d);
\coordinate (J) at (0, 0);

\foreach \p in {A, B, C, D, E, F, G, H, I, J} {
    \fill[blue] (\p) circle (2pt);
}

\draw[red, thick] (A) -- (B) -- (H) -- (J) -- cycle;

\node[above right] at (A) {{\tiny $(2d-1, d-1)$}};
\node[above right] at (B) {{\tiny $(d, 2d)$}};
\node[below right] at (C) {{\tiny $(d-1, d)$}};
\node[below right] at (D) {{\tiny $(3, d-1)$}};
\node[above left] at (E) {{\tiny $(1, 2d-1)$}};
\node[below right] at (F) {{\tiny $(1, d)$}};
\node[below right] at (G) {{\tiny $(1, d-1)$}};
\node[above left] at (H) {{\tiny $(0, 2d+1)$}};
\node[below left] at (I) {{\tiny $(0, d)$}};
\node[below left] at (J) {{\tiny $(0, 0)$}};
\end{tikzpicture}

Thus,
\[v({\rm Newt}(f_d)) = \{(0,d), (0,d+1), (d,-1), (d-1,-d-1), (-2d+1,-d+1)\}\]
This sequence cannot be decomposed as the disjoint union of two subsequences summing to zero. Indeed, suppose there exists $S_1$ and $S_2$, subsets of $v({\rm Newt}(f_d))$, whose elements sum to $(0,0)$, such that $v({\rm Newt}(f_d)) = S_1 \cup S_2$. Looking at the first coordinate, if $(-2d+1,-d+1) \in S_1$, then necessarily $(d,-1),(d-1,-d-1) \in S_1$. Their sum is $(0, -2d-1)$. Thus we must have $(0,d),(0,d+1) \in S_1$ also, and we conclude that $v({\rm Newt}(f_d)) = S_1$. Therefore $f_d$ is irreducible by Lemma \ref{lemma:irreduciblepolynomial} and Remark \ref{irreduciblitycriterion}.
\end{proof}

\begin{prop} \label{invIII} The vector field defined by
$$
    v = (ax^{d}y-cy^{2})\partial_{x}+(ax^{2}y^{d-1}+bx)\partial_{y}
$$
with $d \geq 6$, has a prime $2$-divisor on $\mathbb{A}^2_{\field}$, if $abc \neq 0$.
\end{prop}

\begin{proof} Induction shows that the $2$-divisor is given by $f_{d}(x,y)$ where
\begin{itemize}
    \item if $d\equiv 0 \mod 2$:
        $$
            f_{d}(x,y) = a^{2}bx^{2d}y^{2}+a^{2}bx^{d+3}y^{d-1}+ab^{2}x^{d+2}+a^{2}cx^{4}y^{2d-1}+abcx^{3}y^{d}+bc^{2}y^{4}
        $$
    \item if $d\equiv 1 \mod 2$:
        $$
            f_{d}(x,y) = a^{3}x^{2d+1}y^{d+1}+a^{3}x^{d+4}y^{2d-2}+ab^{2}x^{d+2}+a^{2}cx^{d+1}y^{d+2}+abcx^{d}y^{3}+bc^{2}y^{4}
        $$
\end{itemize}

We need to show that $f_{d}(x,y) \in \field[x,y]$ is irreducible. If $d \equiv 0 \mod{2}$, then
\[{\rm Newt}(f_d) = \langle (2d, 2), (d+3,d-1), (d+2,0), (4,2d-1), (3,d), (0,4)\rangle\]
which is graphically indicated as follows (for $d$ big enough):

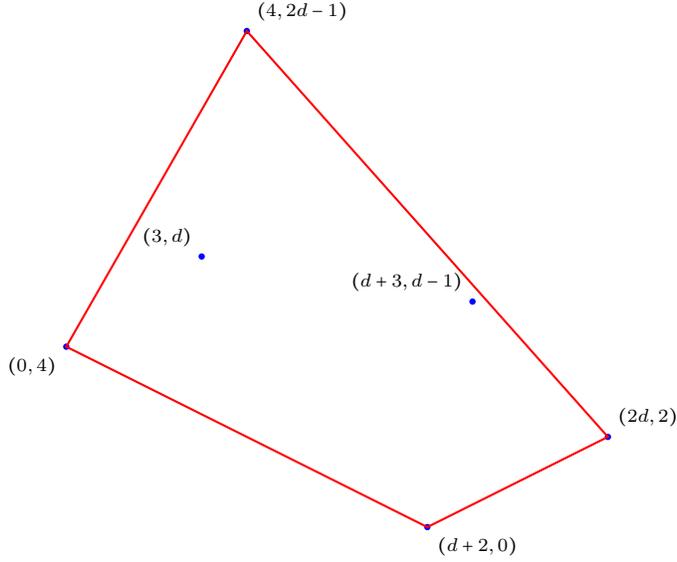
\begin{tikzpicture}[scale=0.6]

\def\d{6} 

\coordinate (A) at (2*\d, 2);
\coordinate (B) at (\d+3, \d-1);
\coordinate (C) at (\d+2, 0);
\coordinate (D) at (4, 2*\d-1);
\coordinate (E) at (3, \d);
\coordinate (F) at (0, 4);

\foreach \p in {A, B, C, D, E, F} {
    \fill[blue] (\p) circle (2pt);
}

\draw[red, thick] (F) -- (D) -- (A) -- (C) -- cycle;

\node[above right] at (A) {{\tiny $(2d, 2)$}};
\node[above left] at (B) {{\tiny $(d+3, d-1)$}};
\node[below right] at (C) {{\tiny $(d+2, 0)$}};
\node[above right] at (D) {{\tiny $(4, 2d-1)$}};
\node[above left] at (E) {{\tiny $(3, d)$}};
\node[below left] at (F) {{\tiny $(0, 4)$}};
\end{tikzpicture}

Thus,
\[v({\rm Newt}(f_d)) = \{(4,2d-5), (2d-4,-2d+3), (-d+2,-2), (-d-2,4)\}.\]

Likewise, if $d \equiv 1 \mod{2}$, then 
\[{\rm Newt}(f_d) = \langle (2d+1,d+1), (d+4,2d-2), (d+2,0), (d+1,d+2), (d,3), (0,4)\rangle\]
which is graphically indicated as follows (for $d$ big enough):

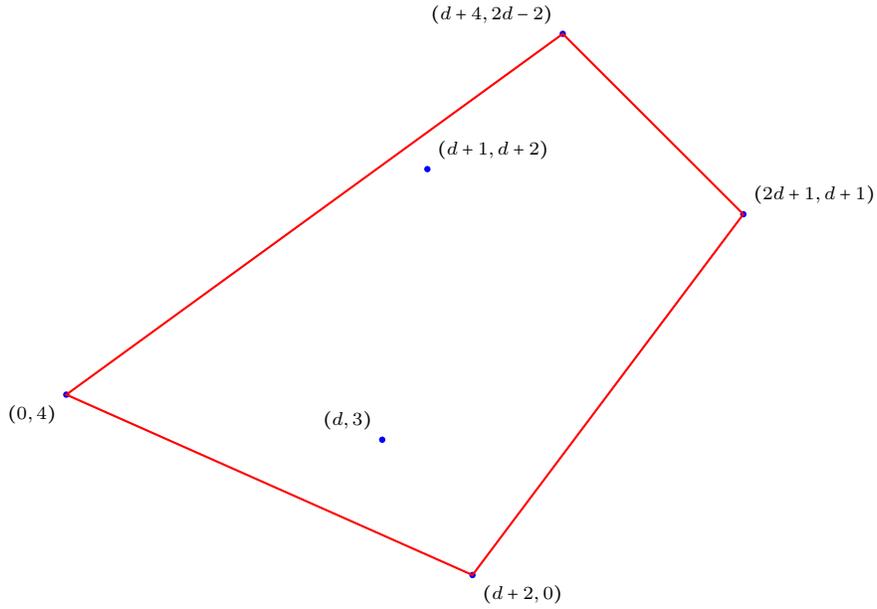
\begin{tikzpicture}[scale=0.6]

\def\d{7} 

\coordinate (A) at (2*\d+1, \d+1);
\coordinate (B) at (\d+4, 2*\d-2);
\coordinate (C) at (\d+2, 0);
\coordinate (D) at (\d+1, \d+2);
\coordinate (E) at (\d, 3);
\coordinate (F) at (0, 4);

\foreach \p in {A, B, C, D, E, F} {
    \fill[blue] (\p) circle (2pt);
}

\draw[red, thick] (F) --  (B) -- (A) -- (C) -- cycle;

\node[above right] at (A) {{\tiny $(2d+1, d+1)$}};
\node[above left] at (B) {{\tiny $(d+4, 2d-2)$}};
\node[below right] at (C) {{\tiny $(d+2, 0)$}};
\node[above right] at (D) {{\tiny $(d+1, d+2)$}};
\node[above left] at (E) {{\tiny $(d, 3)$}};
\node[below left] at (F) {{\tiny $(0, 4)$}};
\end{tikzpicture}

Thus,
\[v({\rm Newt}(f_d)) = \{(d+4, 2d-6), (d-3, -d+3), (-d+1,-d-1), (-d-2,4)\}.\]

In both cases, $v({\rm Newt}(f_d))$ cannot be decomposed into two subsequences summing to zero. Indeed, suppose that $v({\rm Newt}(f_d)) = S_1 \cup S_2$, for two subsets $S_1$ and $S_2$ of $v({\rm Newt}(f_d))$ whose elements sum to $(0,0)$. In both cases, looking only to the first coordinates, the only possibility is that all elements belong to either $S_1$ or $S_2$. Therefore, $f_d$ is irreducible by Lemma \ref{lemma:irreduciblepolynomial} and Remark \ref{irreduciblitycriterion}.
\end{proof}

Consider the following foliations on the complex projective plane defined by the vector fields of the previous two Propositions: 
$$
    \mathcal{F}_e(a,b,c)\colon \quad v = (ax^{e}y-cy^{2})\partial_{x}+(ax^{2}y^{e-1}+bx)\partial_{y}
$$  
$$
    \mathcal{G}_d(u,a,b,c)\colon \quad v = (u+xy^{d})\partial_x+(a+bx+cx^{d-1}+y^{d+1})\partial_{y}
$$
with $a,b,c,u \in \mathbb{Z}$ and $abcu \not\equiv 0 \mod 2$. Lemma \ref{lineinf} ensures that $\mathcal{F}_e(a,b,c)$ has degree $e+1$ and $\mathcal{G}_d(u,a,b,c)$ has degree $d$. Moreover the line $l_{\infty}$ is invariant only by the foliation defined by the extension of $\mathcal{F}_e(a,b,c)$  to $\mathbb{P}_{\mathbb{C}}^{2}$.

\begin{thm} \label{thm:examples} On the complex projective plane, the foliation $\mathcal{F}_e(a,b,c)$ has $l_{\infty}$ as the unique algebraic invariant curve and $\mathcal{G}_d(u,a,b,c)$ does not have algebraic invariant curves.
\end{thm}
\begin{proof} Note that the foliations $\mathcal{F}_e(a,b,c)$ and $\mathcal{G}_d(u,a,b,c)$ are nondicritical. Note that  $\mathcal{F}_e(a,b,c)$ and $\mathcal{G}_d(u, a,b,c)$ have good reduction at $ p =2 $, since $abcu \not\equiv 0 \mod 2$. By Proposition \ref{invII} and Proposition \ref{invIII} we know that the $2$-divisor of those foliations is irreducible on $D_{+}(z) = \mathbb{A}_{\field}^{2}$. Therefore, by Theorem \ref{main}, we conclude that $\mathcal{F}_e(a,b,c)$ and $\mathcal{G}_d(u, a,b,c)$ do not have algebraic invariant curves in $D_{+}(z)$. Lemma \ref{lineinf} finishes the proof.
\end{proof}

\section{Appendix: Code to compute the \texorpdfstring{$p$}{p}-divisor (Singular)} \label{appe}

The following code runs in Singular \cite{DGPS}.  
Given an effective field $K$ (i.e. $K$ is a field for which the operations of sum, subtraction, multiplication and division can be implemented in a computer; see \cite{coutinho2018bounding}) of characteristic $p>0$, and a polynomial vector field 
$$v = A(x,y)\partial_{x}+B(x,y)\partial_{y}$$
where $A(x,y), B(x,y) \in K[x,y]$, the command \verb+pcampo(p,A,B)+ computes the vector field $v^{p}$. The command \verb+pdiv(p,A,B)+ computes the $p$-divisor, that is, the polynomial given by 
$$
   \frac{v\wedge v^{p}}{\partial_x \wedge\partial_y}
$$

\vspace{0.2cm}

\begin{verbatim}
proc pcampo(int p, poly A, poly B){
    int i; 
    matrix AA[p][1]; 
    matrix BB[p][1]; 
    matrix v[2][1]; 
    BB[1,1] = B; 
    AA[1,1] = A; 
    
    for(i=1;i<p;i++){
        AA[i+1,1] = A*diff(AA[i,1],X)+B*diff(AA[i,1],Y);
        BB[i+1,1] = A*diff(BB[i,1],X)+B*diff(BB[i,1],Y);
    }

    v[1,1] = AA[p,1];
    v[2,1] = BB[p,1];
    return (v);
}


proc pdiv(int p, poly A, poly B){
    poly U, V;
    U,V = pcampo(p,A,B);
    return (A*pcampo(p,A,B)[2,1]- B*pcampo(p,A,B)[1,1]);
}
\end{verbatim}

\noindent\textbf{Acknowledgements.}  W. Mendson acknowledges the
support of Capes, CNPq and Universidade Federal Fluminense (UFF). J. P. Figueredo acknowledges the support of the Instituto de Matemática e Estatística (Mathematics and Statistics Institute) of Universidade Federal Fluminense (Fluminense Federal University).

\bibliographystyle{alpha}

\bibliography{annot}

\end{document}